\documentclass{amsart}
\usepackage{amsmath}
\usepackage{amssymb}
\usepackage{amsfonts}

\newtheorem{theorem}{Theorem}
\theoremstyle{plain}

\newtheorem{corollary}{Corollary}

\newtheorem{definition}{Definition}

\newtheorem{lemma}{Lemma}

\numberwithin{equation}{section}
\input{tcilatex}

\begin{document}
\title[Some Integral Inequalities]{ON SOME HADAMARD-TYPE INEQUALITIES FOR
CO-ORDINATED\ CONVEX FUNCTIONS}
\author{M. Emin \"{O}zdemir$^{\blacklozenge }$}
\address{$^{\blacklozenge }$Ataturk University, K.K. Education Faculty,
Department of Mathematics, 25240, Erzurum, Turkey}
\email{emos@atauni.edu.tr}
\author{Ahmet Ocak Akdemir$^{\spadesuit }$}
\address{$^{\spadesuit }$A\u{g}r\i\ \.{I}brahim \c{C}e\c{c}en University,
Faculty of Science and Arts, Department of Mathematics, 04100, A\u{g}r\i ,
Turkey}
\email{ahmetakdemir@agri.edu.tr}
\author{$^{\blacksquare }$Mevl\"{u}t TUN\c{C}}
\address{$^{\blacksquare }$Kilis 7 Aral\i k University, Faculty of Science
and Arts, Department of Mathematics, 79000, Kilis, Turkey}
\email{mevlutunc@kilis.edu.tr}

\begin{abstract}
In this paper, we prove some new inequalities of Hadamard-type for convex
functions on the co-ordinates.
\end{abstract}

\maketitle

\section{INTRODUCTION}

Let $f:I\subseteq 
%TCIMACRO{\U{211d} }%
%BeginExpansion
\mathbb{R}
%EndExpansion
\rightarrow 
%TCIMACRO{\U{211d} }%
%BeginExpansion
\mathbb{R}
%EndExpansion
$ be a convex function defined on the interval $I$ of real numbers and $a<b.$
The following double inequality;%
\begin{equation*}
f\left( \frac{a+b}{2}\right) \leq \frac{1}{b-a}\dint\limits_{a}^{b}f(x)dx%
\leq \frac{f(a)+f(b)}{2}
\end{equation*}

is well known in the literature as Hadamard's inequality. Both inequalities
hold in the reversed direction if $f$ is concave.

In \cite{SS}, Dragomir defined convex functions on the co-ordinates as
following:

\begin{definition}
Let us consider the bidimensional interval $\Delta =[a,b]\times \lbrack c,d]$
in $%
%TCIMACRO{\U{211d} }%
%BeginExpansion
\mathbb{R}
%EndExpansion
^{2}$ with $a<b,$ $c<d.$ A function $f:\Delta \rightarrow 
%TCIMACRO{\U{211d} }%
%BeginExpansion
\mathbb{R}
%EndExpansion
$ will be called convex on the co-ordinates if the partial mappings $%
f_{y}:[a,b]\rightarrow 
%TCIMACRO{\U{211d} }%
%BeginExpansion
\mathbb{R}
%EndExpansion
,$ $f_{y}(u)=f(u,y)$ and $f_{x}:[c,d]\rightarrow 
%TCIMACRO{\U{211d} }%
%BeginExpansion
\mathbb{R}
%EndExpansion
,$ $f_{x}(v)=f(x,v)$ are convex where defined for all $y\in \lbrack c,d]$
and $x\in \lbrack a,b].$ Recall that the mapping $f:\Delta \rightarrow 
%TCIMACRO{\U{211d} }%
%BeginExpansion
\mathbb{R}
%EndExpansion
$ is convex on $\Delta $ if the following inequality holds, 
\begin{equation*}
f(\lambda x+(1-\lambda )z,\lambda y+(1-\lambda )w)\leq \lambda
f(x,y)+(1-\lambda )f(z,w)
\end{equation*}%
for all $(x,y),(z,w)\in \Delta $ and $\lambda \in \lbrack 0,1].$
\end{definition}

In \cite{SS}, Dragomir established the following inequalities of Hadamard's
type for co-ordinated convex functions on a rectangle from the plane $%
%TCIMACRO{\U{211d} }%
%BeginExpansion
\mathbb{R}
%EndExpansion
^{2}.$

\begin{theorem}
Suppose that $f:\Delta =[a,b]\times \lbrack c,d]\rightarrow 
%TCIMACRO{\U{211d} }%
%BeginExpansion
\mathbb{R}
%EndExpansion
$ is convex on the co-ordinates on $\Delta $. Then one has the inequalities;%
\begin{eqnarray}
&&\ f(\frac{a+b}{2},\frac{c+d}{2})  \label{1.1} \\
&\leq &\frac{1}{2}\left[ \frac{1}{b-a}\int_{a}^{b}f(x,\frac{c+d}{2})dx+\frac{%
1}{d-c}\int_{c}^{d}f(\frac{a+b}{2},y)dy\right]  \notag \\
&\leq &\frac{1}{(b-a)(d-c)}\int_{a}^{b}\int_{c}^{d}f(x,y)dxdy  \notag \\
&\leq &\frac{1}{4}\left[ \frac{1}{(b-a)}\int_{a}^{b}f(x,c)dx+\frac{1}{(b-a)}%
\int_{a}^{b}f(x,d)dx\right.  \notag \\
&&\left. +\frac{1}{(d-c)}\int_{c}^{d}f(a,y)dy+\frac{1}{(d-c)}%
\int_{c}^{d}f(b,y)dy\right]  \notag \\
&\leq &\frac{f(a,c)+f(a,d)+f(b,c)+f(b,d)}{4}.  \notag
\end{eqnarray}%
The above inequalities are sharp.
\end{theorem}

Similar results can be found in \cite{SS}-\cite{B}.

The main purpose of this paper is to prove some new inequalities of
Hadamard-type for convex functions on the co-ordinates.

\section{MAIN RESULTS}

To prove our main result, we need the following Lemma.

\begin{lemma}
\label{l1} Let $f:\Delta =\left[ a,b\right] \times \left[ c,d\right]
\rightarrow 
%TCIMACRO{\U{211d} }%
%BeginExpansion
\mathbb{R}
%EndExpansion
$ be a twice partial differentiable mapping on $\Delta =\left[ a,b\right]
\times \left[ c,d\right] .$ If $\frac{\partial ^{2}f}{\partial t\partial s}%
\in L\left( \Delta \right) ,$ then the following equality holds:%
\begin{eqnarray*}
&&A+\frac{1}{\left( b-a\right) \left( d-c\right) }\int\limits_{a}^{b}\int%
\limits_{c}^{d}f\left( u,v\right) dudv \\
&=&\frac{\left( x-a\right) ^{2}\left( y-c\right) ^{2}}{\left( b-a\right)
\left( d-c\right) }\int\limits_{0}^{1}\int\limits_{0}^{1}\left( t-1\right)
\left( s-1\right) \frac{\partial ^{2}f}{\partial t\partial s}\left(
tx+\left( 1-t\right) a,sy+\left( 1-s\right) c\right) dsdt \\
&&+\frac{\left( x-a\right) ^{2}\left( d-y\right) ^{2}}{\left( b-a\right)
\left( d-c\right) }\int\limits_{0}^{1}\int\limits_{0}^{1}\left( t-1\right)
\left( 1-s\right) \frac{\partial ^{2}f}{\partial t\partial s}\left(
tx+\left( 1-t\right) a,sy+\left( 1-s\right) d\right) dsdt \\
&&+\frac{\left( b-x\right) ^{2}\left( y-c\right) ^{2}}{\left( b-a\right)
\left( d-c\right) }\int\limits_{0}^{1}\int\limits_{0}^{1}\left( 1-t\right)
\left( s-1\right) \frac{\partial ^{2}f}{\partial t\partial s}\left(
tx+\left( 1-t\right) b,sy+\left( 1-s\right) c\right) dsdt \\
&&+\frac{\left( b-x\right) ^{2}\left( d-y\right) ^{2}}{\left( b-a\right)
\left( d-c\right) }\int\limits_{0}^{1}\int\limits_{0}^{1}\left( 1-t\right)
\left( 1-s\right) \frac{\partial ^{2}f}{\partial t\partial s}\left(
tx+\left( 1-t\right) b,sy+\left( 1-s\right) d\right) dsdt
\end{eqnarray*}%
where%
\begin{eqnarray*}
A= &&\frac{\left( x-a\right) \left( y-c\right) f\left( a,c\right) +\left(
x-a\right) \left( d-y\right) f\left( a,d\right) }{\left( b-a\right) \left(
d-c\right) } \\
&&+\frac{\left( b-x\right) \left( y-c\right) f\left( b,c\right) +\left(
b-x\right) \left( d-y\right) f\left( b,d\right) }{\left( b-a\right) \left(
d-c\right) } \\
&&-\frac{\left( x-a\right) }{\left( b-a\right) \left( d-c\right) }%
\int\limits_{c}^{d}f\left( a,v\right) dv-\frac{\left( b-x\right) }{\left(
b-a\right) \left( d-c\right) }\int\limits_{c}^{d}f\left( b,v\right) dv \\
&&-\frac{\left( d-y\right) }{\left( b-a\right) \left( d-c\right) }%
\int\limits_{a}^{b}f\left( u,d\right) du-\frac{\left( y-c\right) }{\left(
b-a\right) \left( d-c\right) }\int\limits_{a}^{b}f\left( u,c\right) du
\end{eqnarray*}
\end{lemma}

\begin{proof}
It suffices to note that%
\begin{eqnarray*}
I &=&\underset{I_{1}}{\underbrace{\frac{\left( x-a\right) ^{2}\left(
y-c\right) ^{2}}{\left( b-a\right) \left( d-c\right) }\int\limits_{0}^{1}%
\int\limits_{0}^{1}\left( t-1\right) \left( s-1\right) \frac{\partial ^{2}f}{%
\partial t\partial s}\left( tx+\left( 1-t\right) a,sy+\left( 1-s\right)
c\right) dsdt}} \\
&&+\underset{I_{2}}{\underbrace{\frac{\left( x-a\right) ^{2}\left(
d-y\right) ^{2}}{\left( b-a\right) \left( d-c\right) }\int\limits_{0}^{1}%
\int\limits_{0}^{1}\left( t-1\right) \left( 1-s\right) \frac{\partial ^{2}f}{%
\partial t\partial s}\left( tx+\left( 1-t\right) a,sy+\left( 1-s\right)
d\right) dsdt}} \\
&&+\underset{I_{3}}{\underbrace{\frac{\left( b-x\right) \left( y-c\right) }{%
\left( b-a\right) \left( d-c\right) }\int\limits_{0}^{1}\int\limits_{0}^{1}%
\left( 1-t\right) \left( s-1\right) \frac{\partial ^{2}f}{\partial t\partial
s}\left( tx+\left( 1-t\right) b,sy+\left( 1-s\right) c\right) dsdt}} \\
&&+\underset{I_{4}}{\underbrace{\frac{\left( b-x\right) \left( d-y\right) }{%
\left( b-a\right) \left( d-c\right) }\int\limits_{0}^{1}\int\limits_{0}^{1}%
\left( 1-t\right) \left( 1-s\right) \frac{\partial ^{2}f}{\partial t\partial
s}\left( tx+\left( 1-t\right) b,sy+\left( 1-s\right) d\right) dsdt}}.
\end{eqnarray*}%
Integrating by parts, we get%
\begin{eqnarray*}
I_{1} &=&\frac{\left( x-a\right) ^{2}\left( y-c\right) ^{2}}{\left(
b-a\right) \left( d-c\right) }\int\limits_{0}^{1}\left( s-1\right) \left[ 
\frac{t-1}{x-a}\frac{\partial f}{\partial s}\left( tx+\left( 1-t\right)
a,sy+\left( 1-s\right) c\right) \right. _{0}^{1} \\
&&\left. -\frac{1}{x-a}\int\limits_{0}^{1}\frac{\partial f}{\partial s}%
\left( tx+\left( 1-t\right) a,sy+\left( 1-s\right) c\right) dt\right] ds \\
&=&\frac{\left( x-a\right) ^{2}\left( y-c\right) ^{2}}{\left( b-a\right)
\left( d-c\right) }\int\limits_{0}^{1}\left( s-1\right) \left[ \frac{1}{x-a}%
\frac{\partial f}{\partial s}\left( a,sy+\left( 1-s\right) c\right) \right.
\\
&&\left. -\frac{1}{x-a}\int\limits_{0}^{1}\frac{\partial f}{\partial s}%
\left( tx+\left( 1-t\right) a,sy+\left( 1-s\right) c\right) dt\right] ds
\end{eqnarray*}%
By integrating again and by changing of the variables $u=tx+\left(
1-t\right) a$, $v=sy+\left( 1-s\right) c$, we obtain%
\begin{eqnarray*}
I_{1} &=&\frac{\left( x-a\right) \left( y-c\right) }{\left( b-a\right)
\left( d-c\right) }\int\limits_{0}^{1}\left( s-1\right) \left[ \frac{t-1}{x-a%
}\frac{\partial f}{\partial s}\left( tx+\left( 1-t\right) a,sy+\left(
1-s\right) c\right) \right. _{0}^{1} \\
&&\left. -\frac{1}{x-a}\int\limits_{0}^{1}\frac{\partial f}{\partial s}%
\left( tx+\left( 1-t\right) a,sy+\left( 1-s\right) c\right) dt\right] ds \\
&=&\frac{1}{\left( x-a\right) \left( y-c\right) }f\left( a,c\right) -\frac{1%
}{\left( x-a\right) \left( y-c\right) ^{2}}\int\limits_{c}^{y}f\left(
a,v\right) dv \\
&&-\frac{1}{\left( x-a\right) ^{2}\left( y-c\right) }\int\limits_{a}^{x}f%
\left( u,c\right) du+\frac{1}{\left( x-a\right) ^{2}\left( y-c\right) ^{2}}%
\int\limits_{a}^{x}\int\limits_{c}^{y}f\left( u,v\right) dudv.
\end{eqnarray*}%
By a similar argument, we have 
\begin{eqnarray*}
I_{2} &=&\frac{1}{\left( x-a\right) \left( d-y\right) }f\left( a,d\right) -%
\frac{1}{\left( x-a\right) \left( d-y\right) ^{2}}\int\limits_{y}^{d}f\left(
a,v\right) dv \\
&&-\frac{1}{\left( x-a\right) ^{2}\left( d-y\right) }\int\limits_{a}^{x}f%
\left( u,d\right) du \\
&&+\frac{1}{\left( x-a\right) ^{2}\left( d-y\right) ^{2}}\int\limits_{a}^{x}%
\int\limits_{y}^{d}f\left( u,v\right) dudv,
\end{eqnarray*}%
\begin{eqnarray*}
I_{3} &=&\frac{1}{\left( b-x\right) \left( y-c\right) }f\left( b,c\right) -%
\frac{1}{\left( b-x\right) \left( y-c\right) ^{2}}\int\limits_{c}^{y}f\left(
b,v\right) dv \\
&&-\frac{1}{\left( b-x\right) ^{2}\left( y-c\right) }\int\limits_{x}^{b}f%
\left( u,c\right) du \\
&&+\frac{1}{\left( b-x\right) ^{2}\left( y-c\right) ^{2}}\int\limits_{x}^{b}%
\int\limits_{c}^{y}f\left( u,v\right) dudv
\end{eqnarray*}%
and%
\begin{eqnarray*}
I_{4} &=&\frac{1}{\left( b-x\right) \left( d-y\right) }f\left( b,d\right) -%
\frac{1}{\left( b-x\right) \left( d-y\right) ^{2}}\int\limits_{y}^{d}f\left(
b,v\right) dv \\
&&-\frac{1}{\left( b-x\right) ^{2}\left( d-y\right) }\int\limits_{x}^{b}f%
\left( u,d\right) du \\
&&+\frac{1}{\left( b-x\right) ^{2}\left( d-y\right) ^{2}}\int\limits_{x}^{b}%
\int\limits_{y}^{d}f\left( u,v\right) dudv.
\end{eqnarray*}%
Therefore, we obtain%
\begin{eqnarray*}
&&I_{1}+I_{2}+I_{3}+I_{4} \\
&=&\frac{1}{\left( b-a\right) \left( d-c\right) } \\
&&\times \left[ A-\left( x-a\right) \left[ \int\limits_{y}^{d}f\left(
a,v\right) dv+\int\limits_{c}^{y}f\left( a,v\right) dv\right] -\left(
b-x\right) \left[ \int\limits_{c}^{y}f\left( b,v\right)
dv+\int\limits_{y}^{d}f\left( b,v\right) dv\right] \right. \\
&&-\left( d-y\right) \left[ \int\limits_{a}^{x}f\left( u,d\right)
du+\int\limits_{x}^{b}f\left( u,d\right) du\right] -\left( y-c\right) \left[
\int\limits_{a}^{x}f\left( u,c\right) du+\int\limits_{x}^{b}f\left(
u,c\right) du\right] \\
&&+\int\limits_{x}^{b}\int\limits_{c}^{y}f\left( u,v\right)
dudv+\int\limits_{x}^{b}\int\limits_{y}^{d}f\left( u,v\right)
dudv+\int\limits_{a}^{x}\int\limits_{c}^{y}f\left( u,v\right)
dudv+\int\limits_{a}^{x}\int\limits_{y}^{d}f\left( u,v\right) dudv \\
&=&\frac{1}{\left( b-a\right) \left( d-c\right) }\left[ A-\left( x-a\right)
\int\limits_{c}^{d}f\left( a,v\right) dv-\left( b-x\right)
\int\limits_{c}^{d}f\left( b,v\right) dv\right. \\
&&-\left( d-y\right) \int\limits_{a}^{b}f\left( u,d\right) du-\left(
y-c\right) \int\limits_{a}^{b}f\left( u,c\right)
du+\int\limits_{a}^{b}\int\limits_{c}^{d}f\left( u,v\right) dudv.
\end{eqnarray*}

Which completes the proof.
\end{proof}

\begin{theorem}
\label{t1}Let $f:\Delta =\left[ a,b\right] \times \left[ c,d\right]
\rightarrow 
%TCIMACRO{\U{211d} }%
%BeginExpansion
\mathbb{R}
%EndExpansion
$ be a partial differentiable mapping on $\Delta =\left[ a,b\right] \times %
\left[ c,d\right] $ and $\frac{\partial ^{2}f}{\partial t\partial s}\in
L\left( \Delta \right) $. If $\left\vert \frac{\partial ^{2}f}{\partial
t\partial s}\right\vert $ is a convex function on the co-ordinates on $%
\Delta ,$ then the following inequality holds;%
\begin{eqnarray*}
&&\left\vert A+\frac{1}{\left( b-a\right) \left( d-c\right) }%
\int\limits_{a}^{b}\int\limits_{c}^{d}f\left( u,v\right) dudv\right\vert \\
&\leq &\frac{1}{9\left( b-a\right) \left( d-c\right) }\left[ \left( \frac{%
\left( \left( x-a\right) ^{2}+\left( b-x\right) ^{2}\right) \left( \left(
y-c\right) ^{2}+\left( d-y\right) ^{2}\right) }{4}\right) \left\vert \frac{%
\partial ^{2}f}{\partial t\partial s}\left( x,y\right) \right\vert \right. \\
&&+\left( \frac{\left( x-a\right) ^{2}\left( \left( y-c\right) ^{2}+\left(
d-y\right) ^{2}\right) }{2}\right) \left\vert \frac{\partial ^{2}f}{\partial
t\partial s}\left( a,y\right) \right\vert \\
&&+\left( \frac{\left( b-x\right) ^{2}\left( \left( y-c\right) ^{2}+\left(
d-y\right) ^{2}\right) }{2}\right) \left\vert \frac{\partial ^{2}f}{\partial
t\partial s}\left( b,y\right) \right\vert \\
&&+\left( \frac{\left( y-c\right) ^{2}\left( \left( x-a\right) ^{2}+\left(
b-x\right) ^{2}\right) }{2}\right) \left\vert \frac{\partial ^{2}f}{\partial
t\partial s}\left( x,c\right) \right\vert \\
&&+\left( \frac{\left( d-y\right) ^{2}\left( \left( x-a\right) ^{2}+\left(
b-x\right) ^{2}\right) }{2}\right) \left\vert \frac{\partial ^{2}f}{\partial
t\partial s}\left( x,d\right) \right\vert \\
&&+\left( x-a\right) ^{2}\left( y-c\right) ^{2}\left\vert \frac{\partial
^{2}f}{\partial t\partial s}\left( a,c\right) \right\vert +\left( x-a\right)
^{2}\left( d-y\right) ^{2}\left\vert \frac{\partial ^{2}f}{\partial
t\partial s}\left( a,d\right) \right\vert \\
&&\left. +\left( b-x\right) ^{2}\left( y-c\right) ^{2}\left\vert \frac{%
\partial ^{2}f}{\partial t\partial s}\left( b,c\right) \right\vert +\left(
b-x\right) ^{2}\left( d-y\right) ^{2}\left\vert \frac{\partial ^{2}f}{%
\partial t\partial s}\left( b,d\right) \right\vert \right] .
\end{eqnarray*}
\end{theorem}

\begin{proof}
From Lemma \ref{l1} and using the property of modulus, we have%
\begin{eqnarray*}
&&\left\vert A+\frac{1}{\left( b-a\right) \left( d-c\right) }%
\int\limits_{a}^{b}\int\limits_{c}^{d}f\left( u,v\right) dudv\right\vert \\
&\leq &\frac{\left( x-a\right) ^{2}\left( y-c\right) ^{2}}{\left( b-a\right)
\left( d-c\right) }\int\limits_{0}^{1}\int\limits_{0}^{1}\left\vert \left(
t-1\right) \left( s-1\right) \right\vert \left\vert \frac{\partial ^{2}f}{%
\partial t\partial s}\left( tx+\left( 1-t\right) a,sy+\left( 1-s\right)
c\right) \right\vert dsdt \\
&&+\frac{\left( x-a\right) ^{2}\left( d-y\right) ^{2}}{\left( b-a\right)
\left( d-c\right) }\int\limits_{0}^{1}\int\limits_{0}^{1}\left\vert \left(
t-1\right) \left( 1-s\right) \right\vert \left\vert \frac{\partial ^{2}f}{%
\partial t\partial s}\left( tx+\left( 1-t\right) a,sy+\left( 1-s\right)
d\right) \right\vert dsdt \\
&&+\frac{\left( b-x\right) ^{2}\left( y-c\right) ^{2}}{\left( b-a\right)
\left( d-c\right) }\int\limits_{0}^{1}\int\limits_{0}^{1}\left\vert \left(
1-t\right) \left( s-1\right) \right\vert \left\vert \frac{\partial ^{2}f}{%
\partial t\partial s}\left( tx+\left( 1-t\right) b,sy+\left( 1-s\right)
c\right) \right\vert dsdt \\
&&+\frac{\left( b-x\right) ^{2}\left( d-y\right) ^{2}}{\left( b-a\right)
\left( d-c\right) }\int\limits_{0}^{1}\int\limits_{0}^{1}\left\vert \left(
1-t\right) \left( 1-s\right) \right\vert \left\vert \frac{\partial ^{2}f}{%
\partial t\partial s}\left( tx+\left( 1-t\right) b,sy+\left( 1-s\right)
d\right) \right\vert dsdt.
\end{eqnarray*}%
Since $\left\vert \frac{\partial ^{2}f}{\partial t\partial s}\right\vert $
is co-ordinated convex, we can write%
\begin{eqnarray*}
&&\left\vert A+\frac{1}{\left( b-a\right) \left( d-c\right) }%
\int\limits_{a}^{b}\int\limits_{c}^{d}f\left( u,v\right) dudv\right\vert \\
&\leq &\frac{\left( x-a\right) ^{2}\left( y-c\right) ^{2}}{\left( b-a\right)
\left( d-c\right) }\int\limits_{0}^{1}\left\vert \left( s-1\right)
\right\vert \\
&&\times \left[ \int\limits_{0}^{1}\left( t-1\right) t\left\vert \frac{%
\partial ^{2}f}{\partial t\partial s}\left( x,sy+\left( 1-s\right) c\right)
\right\vert dt\right. \\
&&\left. +\int\limits_{0}^{1}\left( t-1\right) \left( 1-t\right) \left\vert 
\frac{\partial ^{2}f}{\partial t\partial s}\left( a,sy+\left( 1-s\right)
c\right) \right\vert dt\right] ds \\
&&+\frac{\left( x-a\right) ^{2}\left( d-y\right) ^{2}}{\left( b-a\right)
\left( d-c\right) }\int\limits_{0}^{1}\left\vert \left( 1-s\right)
\right\vert \\
&&\times \left[ \int\limits_{0}^{1}\left( t-1\right) t\left\vert \frac{%
\partial ^{2}f}{\partial t\partial s}\left( x,sy+\left( 1-s\right) d\right)
\right\vert dt\right. \\
&&\left. +\int\limits_{0}^{1}\left( t-1\right) \left( 1-t\right) \left\vert 
\frac{\partial ^{2}f}{\partial t\partial s}\left( a,sy+\left( 1-s\right)
d\right) \right\vert dt\right] ds \\
&&+\frac{\left( b-x\right) ^{2}\left( y-c\right) ^{2}}{\left( b-a\right)
\left( d-c\right) }\int\limits_{0}^{1}\left\vert \left( s-1\right)
\right\vert \\
&&\times \left[ \int\limits_{0}^{1}\left( t-1\right) t\left\vert \frac{%
\partial ^{2}f}{\partial t\partial s}\left( x,sy+\left( 1-s\right) c\right)
\right\vert dt\right. \\
&&\left. +\int\limits_{0}^{1}\left( t-1\right) \left( 1-t\right) \left\vert 
\frac{\partial ^{2}f}{\partial t\partial s}\left( b,sy+\left( 1-s\right)
c\right) \right\vert dt\right] ds \\
&&+\frac{\left( b-x\right) ^{2}\left( d-y\right) ^{2}}{\left( b-a\right)
\left( d-c\right) }\int\limits_{0}^{1}\left\vert \left( 1-s\right)
\right\vert \\
&&\times \left[ \int\limits_{0}^{1}\left( 1-t\right) t\left\vert \frac{%
\partial ^{2}f}{\partial t\partial s}\left( x,sy+\left( 1-s\right) d\right)
\right\vert dt\right. \\
&&\left. +\int\limits_{0}^{1}\left( 1-t\right) \left( 1-t\right) \left\vert 
\frac{\partial ^{2}f}{\partial t\partial s}\left( b,sy+\left( 1-s\right)
d\right) \right\vert dt\right] ds.
\end{eqnarray*}%
By computing these integrals, we obtain%
\begin{eqnarray*}
&&\left\vert A+\frac{1}{\left( b-a\right) \left( d-c\right) }%
\int\limits_{a}^{b}\int\limits_{c}^{d}f\left( u,v\right) dudv\right\vert \\
&\leq &\frac{\left( x-a\right) ^{2}\left( y-c\right) ^{2}}{\left( b-a\right)
\left( d-c\right) }\int\limits_{0}^{1}\left\vert s-1\right\vert \left[ -%
\frac{1}{6}\left\vert \frac{\partial ^{2}f}{\partial t\partial s}\left(
x,sy+\left( 1-s\right) c\right) \right\vert -\frac{1}{3}\left\vert \frac{%
\partial ^{2}f}{\partial t\partial s}\left( a,sy+\left( 1-s\right) c\right)
\right\vert \right] ds \\
&&+\frac{\left( x-a\right) ^{2}\left( d-y\right) ^{2}}{\left( b-a\right)
\left( d-c\right) }\int\limits_{0}^{1}\left\vert 1-s\right\vert \left[ -%
\frac{1}{6}\left\vert \frac{\partial ^{2}f}{\partial t\partial s}\left(
x,sy+\left( 1-s\right) d\right) \right\vert -\frac{1}{3}\left\vert \frac{%
\partial ^{2}f}{\partial t\partial s}\left( a,sy+\left( 1-s\right) d\right)
\right\vert \right] ds \\
&&+\frac{\left( b-x\right) ^{2}\left( y-c\right) ^{2}}{\left( b-a\right)
\left( d-c\right) }\int\limits_{0}^{1}\left\vert s-1\right\vert \left[ -%
\frac{1}{6}\left\vert \frac{\partial ^{2}f}{\partial t\partial s}\left(
x,sy+\left( 1-s\right) c\right) \right\vert -\frac{1}{3}\left\vert \frac{%
\partial ^{2}f}{\partial t\partial s}\left( b,sy+\left( 1-s\right) c\right)
\right\vert \right] ds \\
&&+\frac{\left( b-x\right) ^{2}\left( d-y\right) ^{2}}{\left( b-a\right)
\left( d-c\right) }\int\limits_{0}^{1}\left\vert 1-s\right\vert \left[ -%
\frac{1}{6}\left\vert \frac{\partial ^{2}f}{\partial t\partial s}\left(
x,sy+\left( 1-s\right) d\right) \right\vert -\frac{1}{3}\left\vert \frac{%
\partial ^{2}f}{\partial t\partial s}\left( b,sy+\left( 1-s\right) d\right)
\right\vert \right] ds.
\end{eqnarray*}%
Using co-ordinated convexity of $\left\vert \frac{\partial ^{2}f}{\partial
t\partial s}\right\vert $ again and computing all integrals, we obtain%
\begin{eqnarray*}
&&\left\vert A+\frac{1}{\left( b-a\right) \left( d-c\right) }%
\int\limits_{a}^{b}\int\limits_{c}^{d}f\left( u,v\right) dudv\right\vert \\
&\leq &\frac{1}{9\left( b-a\right) \left( d-c\right) }\left[ \left( \frac{%
\left( \left( x-a\right) ^{2}+\left( b-x\right) ^{2}\right) \left( \left(
y-c\right) ^{2}+\left( d-y\right) ^{2}\right) }{4}\right) \left\vert \frac{%
\partial ^{2}f}{\partial t\partial s}\left( x,y\right) \right\vert \right. \\
&&+\left( \frac{\left( x-a\right) ^{2}\left( \left( y-c\right) ^{2}+\left(
d-y\right) ^{2}\right) }{2}\right) \left\vert \frac{\partial ^{2}f}{\partial
t\partial s}\left( a,y\right) \right\vert \\
&&+\left( \frac{\left( b-x\right) ^{2}\left( \left( y-c\right) ^{2}+\left(
d-y\right) ^{2}\right) }{2}\right) \left\vert \frac{\partial ^{2}f}{\partial
t\partial s}\left( b,y\right) \right\vert \\
&&+\left( \frac{\left( y-c\right) ^{2}\left( \left( x-a\right) ^{2}+\left(
b-x\right) ^{2}\right) }{2}\right) \left\vert \frac{\partial ^{2}f}{\partial
t\partial s}\left( x,c\right) \right\vert \\
&&+\left( \frac{\left( d-y\right) ^{2}\left( \left( x-a\right) ^{2}+\left(
b-x\right) ^{2}\right) }{2}\right) \left\vert \frac{\partial ^{2}f}{\partial
t\partial s}\left( x,d\right) \right\vert \\
&&+\left( x-a\right) ^{2}\left( y-c\right) ^{2}\left\vert \frac{\partial
^{2}f}{\partial t\partial s}\left( a,c\right) \right\vert +\left( x-a\right)
^{2}\left( d-y\right) ^{2}\left\vert \frac{\partial ^{2}f}{\partial
t\partial s}\left( a,d\right) \right\vert \\
&&\left. +\left( b-x\right) ^{2}\left( y-c\right) ^{2}\left\vert \frac{%
\partial ^{2}f}{\partial t\partial s}\left( b,c\right) \right\vert +\left(
b-x\right) ^{2}\left( d-y\right) ^{2}\left\vert \frac{\partial ^{2}f}{%
\partial t\partial s}\left( b,d\right) \right\vert \right] .
\end{eqnarray*}%
Which completes the proof.
\end{proof}

\begin{corollary}
\label{c1}(1) Under the assumptions of Theorem \ref{t1}, if we choose $x=a,$ 
$y=c,$ we obtain the following inequality;%
\begin{eqnarray*}
&&\left\vert \frac{f\left( b,d\right) }{\left( b-a\right) \left( d-c\right) }%
-\frac{1}{d-c}\int\limits_{c}^{d}f\left( b,v\right) dv\right. \\
&&\left. -\frac{1}{b-a}f\left( u,d\right) du+\frac{1}{\left( b-a\right)
\left( d-c\right) }\int\limits_{a}^{b}\int\limits_{c}^{d}f\left( u,v\right)
dudv\right\vert \\
&\leq &\frac{1}{36\left( b-a\right) \left( d-c\right) }\left\vert \frac{%
\partial ^{2}f}{\partial t\partial s}\left( a,c\right) \right\vert +\frac{1}{%
9\left( b-a\right) \left( d-c\right) }\left\vert \frac{\partial ^{2}f}{%
\partial t\partial s}\left( b,d\right) \right\vert \\
&&+\frac{1}{18\left( b-a\right) \left( d-c\right) }\left\vert \frac{\partial
^{2}f}{\partial t\partial s}\left( a,d\right) \right\vert +\frac{1}{18\left(
b-a\right) \left( d-c\right) }\left\vert \frac{\partial ^{2}f}{\partial
t\partial s}\left( b,c\right) \right\vert .
\end{eqnarray*}%
(2) Under the assumptions of Theorem \ref{t1}, if we choose $x=b,$ $y=d,$ we
obtain the following inequality;%
\begin{eqnarray*}
&&\left\vert \frac{f\left( a,c\right) }{\left( b-a\right) \left( d-c\right) }%
-\frac{1}{d-c}\int\limits_{c}^{d}f\left( a,v\right) dv\right. \\
&&\left. -\frac{1}{b-a}\int\limits_{a}^{b}f\left( u,c\right) du+\frac{1}{%
\left( b-a\right) \left( d-c\right) }\int\limits_{a}^{b}\int\limits_{c}^{d}f%
\left( u,v\right) dudv\right\vert \\
&\leq &\frac{\left( b-a\right) \left( d-c\right) }{36}\left\vert \frac{%
\partial ^{2}f}{\partial t\partial s}\left( b,d\right) \right\vert +\frac{1}{%
9\left( b-a\right) \left( d-c\right) }\left\vert \frac{\partial ^{2}f}{%
\partial t\partial s}\left( a,c\right) \right\vert \\
&&+\frac{\left( b-a\right) \left( d-c\right) }{18}\left\vert \frac{\partial
^{2}f}{\partial t\partial s}\left( a,d\right) \right\vert +\frac{\left(
b-a\right) \left( d-c\right) }{18}\left\vert \frac{\partial ^{2}f}{\partial
t\partial s}\left( b,c\right) \right\vert .
\end{eqnarray*}%
(3) Under the assumptions of Theorem \ref{t1}, if we choose $x=\frac{a+b}{2}%
, $ $y=\frac{c+d}{2},$ we obtain the following inequality;%
\begin{eqnarray*}
&&\left\vert \frac{f\left( a,c\right) +f\left( a,d\right) +f\left(
b,c\right) +f\left( b,d\right) }{4\left( b-a\right) \left( d-c\right) }-%
\frac{1}{2\left( d-c\right) }\int\limits_{c}^{d}f\left( a,v\right) dv-\frac{1%
}{2\left( d-c\right) }\int\limits_{c}^{d}f\left( b,v\right) dv\right. \\
&&\left. -\frac{1}{2\left( b-a\right) }\int\limits_{a}^{b}f\left( u,d\right)
du-\frac{1}{2\left( b-a\right) }\int\limits_{a}^{b}f\left( u,c\right) du+%
\frac{1}{\left( b-a\right) \left( d-c\right) }\int\limits_{a}^{b}\int%
\limits_{c}^{d}f\left( u,v\right) dudv\right\vert \\
&\leq &\frac{1}{144\left( b-a\right) \left( d-c\right) }\left[ \left\vert 
\frac{\partial ^{2}f}{\partial t\partial s}\left( \frac{a+b}{2},\frac{c+d}{2}%
\right) \right\vert \right. +\left\vert \frac{\partial ^{2}f}{\partial
t\partial s}\left( a,\frac{c+d}{2}\right) \right\vert \\
&&+\left\vert \frac{\partial ^{2}f}{\partial t\partial s}\left( b,\frac{c+d}{%
2}\right) \right\vert +\left\vert \frac{\partial ^{2}f}{\partial t\partial s}%
\left( \frac{a+b}{2},c\right) \right\vert +\left\vert \frac{\partial ^{2}f}{%
\partial t\partial s}\left( \frac{a+b}{2},d\right) \right\vert \\
&&+\left\vert \frac{\partial ^{2}f}{\partial t\partial s}\left( a,c\right)
\right\vert +\left\vert \frac{\partial ^{2}f}{\partial t\partial s}\left(
a,d\right) \right\vert \left. +\left\vert \frac{\partial ^{2}f}{\partial
t\partial s}\left( b,c\right) \right\vert +\left\vert \frac{\partial ^{2}f}{%
\partial t\partial s}\left( b,d\right) \right\vert \right] .
\end{eqnarray*}
\end{corollary}

\begin{theorem}
\label{t2}\bigskip Let $f:\Delta =\left[ a,b\right] \times \left[ c,d\right]
\rightarrow 
%TCIMACRO{\U{211d} }%
%BeginExpansion
\mathbb{R}
%EndExpansion
$ be a partial differentiable mapping on $\Delta =\left[ a,b\right] \times %
\left[ c,d\right] $ and $\frac{\partial ^{2}f}{\partial t\partial s}\in
L\left( \Delta \right) $. If $\left\vert \frac{\partial ^{2}f}{\partial
t\partial s}\right\vert ^{q},q>1,$ is a convex function on the co-ordinates
on $\Delta ,$ then the following inequality holds;%
\begin{eqnarray*}
&&\left\vert A+\frac{1}{\left( b-a\right) \left( d-c\right) }%
\int\limits_{a}^{b}\int\limits_{c}^{d}f\left( u,v\right) dudv\right\vert \\
&\leq &\frac{1}{2^{\frac{2}{q}}\left( p+1\right) ^{\frac{2}{p}}}\times \\
&&\left\{ \frac{\left( x-a\right) ^{2}\left( y-c\right) ^{2}}{\left(
b-a\right) \left( d-c\right) }\left( \left\vert \frac{\partial ^{2}f}{%
\partial t\partial s}\left( x,y\right) \right\vert ^{q}+\left\vert \frac{%
\partial ^{2}f}{\partial t\partial s}\left( x,c\right) \right\vert
^{q}+\left\vert \frac{\partial ^{2}f}{\partial t\partial s}\left( a,y\right)
\right\vert ^{q}+\left\vert \frac{\partial ^{2}f}{\partial t\partial s}%
\left( a,c\right) \right\vert ^{q}\right) ^{\frac{1}{q}}\right. \\
&&\left. +\frac{\left( x-a\right) ^{2}\left( d-y\right) ^{2}}{\left(
b-a\right) \left( d-c\right) }\left( \left\vert \frac{\partial ^{2}f}{%
\partial t\partial s}\left( x,y\right) \right\vert ^{q}+\left\vert \frac{%
\partial ^{2}f}{\partial t\partial s}\left( x,d\right) \right\vert
^{q}+\left\vert \frac{\partial ^{2}f}{\partial t\partial s}\left( a,y\right)
\right\vert ^{q}+\left\vert \frac{\partial ^{2}f}{\partial t\partial s}%
\left( a,d\right) \right\vert ^{q}\right) ^{\frac{1}{q}}\right. \\
&&\left. +\frac{\left( b-x\right) ^{2}\left( y-c\right) ^{2}}{\left(
b-a\right) \left( d-c\right) }\left( \left\vert \frac{\partial ^{2}f}{%
\partial t\partial s}\left( x,y\right) \right\vert ^{q}+\left\vert \frac{%
\partial ^{2}f}{\partial t\partial s}\left( x,c\right) \right\vert
^{q}+\left\vert \frac{\partial ^{2}f}{\partial t\partial s}\left( b,y\right)
\right\vert ^{q}+\left\vert \frac{\partial ^{2}f}{\partial t\partial s}%
\left( b,c\right) \right\vert ^{q}\right) ^{\frac{1}{q}}\right. \\
&&\left. +\frac{\left( b-x\right) ^{2}\left( d-y\right) ^{2}}{\left(
b-a\right) \left( d-c\right) }\left( \left\vert \frac{\partial ^{2}f}{%
\partial t\partial s}\left( x,y\right) \right\vert ^{q}+\left\vert \frac{%
\partial ^{2}f}{\partial t\partial s}\left( x,d\right) \right\vert
^{q}+\left\vert \frac{\partial ^{2}f}{\partial t\partial s}\left( b,y\right)
\right\vert ^{q}+\left\vert \frac{\partial ^{2}f}{\partial t\partial s}%
\left( b,d\right) \right\vert ^{q}\right) ^{\frac{1}{q}}\right\}
\end{eqnarray*}%
where $p^{-1}+q^{-1}=1.$
\end{theorem}

\begin{proof}
From Lemma \ref{l1}, we have%
\begin{eqnarray*}
&&\left\vert A+\frac{1}{\left( b-a\right) \left( d-c\right) }%
\int\limits_{a}^{b}\int\limits_{c}^{d}f\left( u,v\right) dudv\right\vert \\
&\leq &\frac{\left( x-a\right) ^{2}\left( y-c\right) ^{2}}{\left( b-a\right)
\left( d-c\right) }\int\limits_{0}^{1}\int\limits_{0}^{1}\left\vert \left(
t-1\right) \left( s-1\right) \right\vert \left\vert \frac{\partial ^{2}f}{%
\partial t\partial s}\left( tx+\left( 1-t\right) a,sy+\left( 1-s\right)
c\right) \right\vert dsdt \\
&&+\frac{\left( x-a\right) ^{2}\left( d-y\right) ^{2}}{\left( b-a\right)
\left( d-c\right) }\int\limits_{0}^{1}\int\limits_{0}^{1}\left\vert \left(
t-1\right) \left( 1-s\right) \right\vert \left\vert \frac{\partial ^{2}f}{%
\partial t\partial s}\left( tx+\left( 1-t\right) a,sy+\left( 1-s\right)
d\right) \right\vert dsdt \\
&&+\frac{\left( b-x\right) ^{2}\left( y-c\right) ^{2}}{\left( b-a\right)
\left( d-c\right) }\int\limits_{0}^{1}\int\limits_{0}^{1}\left\vert \left(
1-t\right) \left( s-1\right) \right\vert \left\vert \frac{\partial ^{2}f}{%
\partial t\partial s}\left( tx+\left( 1-t\right) b,sy+\left( 1-s\right)
c\right) \right\vert dsdt \\
&&+\frac{\left( b-x\right) ^{2}\left( d-y\right) ^{2}}{\left( b-a\right)
\left( d-c\right) }\int\limits_{0}^{1}\int\limits_{0}^{1}\left\vert \left(
1-t\right) \left( 1-s\right) \right\vert \left\vert \frac{\partial ^{2}f}{%
\partial t\partial s}\left( tx+\left( 1-t\right) b,sy+\left( 1-s\right)
d\right) \right\vert dsdt.
\end{eqnarray*}%
By using the well known H\"{o}lder inequality for double integrals, then one
has:%
\begin{eqnarray*}
&&\left\vert A+\frac{1}{\left( b-a\right) \left( d-c\right) }%
\int\limits_{a}^{b}\int\limits_{c}^{d}f\left( u,v\right) dudv\right\vert \\
&\leq &\frac{\left( x-a\right) ^{2}\left( y-c\right) ^{2}}{\left( b-a\right)
\left( d-c\right) }\left( \int\limits_{0}^{1}\int\limits_{0}^{1}\left\vert
\left( t-1\right) \left( s-1\right) \right\vert ^{p}dsdt\right) ^{\frac{1}{p}%
} \\
&&\times \left( \int\limits_{0}^{1}\int\limits_{0}^{1}\left\vert \frac{%
\partial ^{2}f}{\partial t\partial s}\left( tx+\left( 1-t\right) a,sy+\left(
1-s\right) c\right) \right\vert ^{q}dsdt\right) ^{\frac{1}{q}} \\
&&+\frac{\left( x-a\right) ^{2}\left( d-y\right) ^{2}}{\left( b-a\right)
\left( d-c\right) }\left( \int\limits_{0}^{1}\int\limits_{0}^{1}\left\vert
\left( t-1\right) \left( 1-s\right) \right\vert ^{p}dsdt\right) ^{\frac{1}{p}%
} \\
&&\times \left( \int\limits_{0}^{1}\int\limits_{0}^{1}\left\vert \frac{%
\partial ^{2}f}{\partial t\partial s}\left( tx+\left( 1-t\right) a,sy+\left(
1-s\right) d\right) \right\vert ^{q}dsdt\right) ^{\frac{1}{q}} \\
&&+\frac{\left( b-x\right) ^{2}\left( y-c\right) ^{2}}{\left( b-a\right)
\left( d-c\right) }\left( \int\limits_{0}^{1}\int\limits_{0}^{1}\left\vert
\left( 1-t\right) \left( s-1\right) \right\vert ^{p}dsdt\right) ^{\frac{1}{p}%
} \\
&&\times \left( \int\limits_{0}^{1}\int\limits_{0}^{1}\left\vert \frac{%
\partial ^{2}f}{\partial t\partial s}\left( tx+\left( 1-t\right) b,sy+\left(
1-s\right) c\right) \right\vert ^{q}dsdt\right) ^{\frac{1}{q}} \\
&&+\frac{\left( b-x\right) ^{2}\left( d-y\right) ^{2}}{\left( b-a\right)
\left( d-c\right) }\left( \int\limits_{0}^{1}\int\limits_{0}^{1}\left\vert
\left( 1-t\right) \left( 1-s\right) \right\vert ^{p}dsdt\right) ^{\frac{1}{p}%
} \\
&&\times \left( \int\limits_{0}^{1}\int\limits_{0}^{1}\left\vert \frac{%
\partial ^{2}f}{\partial t\partial s}\left( tx+\left( 1-t\right) b,sy+\left(
1-s\right) d\right) \right\vert ^{q}dsdt.\right) ^{\frac{1}{q}}
\end{eqnarray*}

Since $\left\vert \frac{\partial ^{2}f}{\partial t\partial s}\right\vert
^{q} $ is convex function on the co-ordinates on $\Delta $, we know that for 
$t\in \left[ 0,1\right] $%
\begin{eqnarray*}
&&\left\vert \frac{\partial ^{2}f}{\partial t\partial s}\left( tx+\left(
1-t\right) a,sy+\left( 1-s\right) c\right) \right\vert ^{q} \\
&\leq &t\left\vert \frac{\partial ^{2}f}{\partial t\partial s}\left(
x,sy+\left( 1-s\right) c\right) \right\vert ^{q}+\left( 1-t\right)
\left\vert \frac{\partial ^{2}f}{\partial t\partial s}\left( a,sy+\left(
1-s\right) c\right) \right\vert ^{q} \\
&\leq &t\left( s\left\vert \frac{\partial ^{2}f}{\partial t\partial s}\left(
x,y\right) \right\vert ^{q}+\left( 1-s\right) \left\vert \frac{\partial ^{2}f%
}{\partial t\partial s}\left( x,c\right) \right\vert ^{q}\right) +\left(
1-t\right) \left( s\left\vert \frac{\partial ^{2}f}{\partial t\partial s}%
\left( a,y\right) \right\vert ^{q}+\left( 1-s\right) \left\vert \frac{%
\partial ^{2}f}{\partial t\partial s}\left( a,c\right) \right\vert
^{q}\right)
\end{eqnarray*}%
and by using the fact that%
\begin{equation*}
\left( \int\limits_{0}^{1}\int\limits_{0}^{1}\left\vert \left( t-1\right)
\left( s-1\right) \right\vert ^{p}dsdt\right) ^{\frac{1}{p}}=\frac{1}{\left(
p+1\right) ^{\frac{2}{p}}}
\end{equation*}%
we get%
\begin{eqnarray}
&&\left( \int\limits_{0}^{1}\int\limits_{0}^{1}\left\vert \frac{\partial
^{2}f}{\partial t\partial s}\left( tx+\left( 1-t\right) a,sy+\left(
1-s\right) c\right) \right\vert ^{q}dsdt\right) ^{\frac{1}{q}}  \label{m2} \\
&\leq &\left( \int\limits_{0}^{1}\int\limits_{0}^{1}\left\{ ts\left\vert 
\frac{\partial ^{2}f}{\partial t\partial s}\left( x,y\right) \right\vert
^{q}+t\left( 1-s\right) \left\vert \frac{\partial ^{2}f}{\partial t\partial s%
}\left( x,c\right) \right\vert ^{q}\right. \right.  \notag \\
&&\left. \left. +\left( 1-t\right) s\left\vert \frac{\partial ^{2}f}{%
\partial t\partial s}\left( a,y\right) \right\vert ^{q}+\left( 1-t\right)
\left( 1-s\right) \left\vert \frac{\partial ^{2}f}{\partial t\partial s}%
\left( a,c\right) \right\vert ^{q}\right\} dtds\right) ^{\frac{1}{q}}  \notag
\\
&=&\frac{1}{2^{\frac{2}{q}}}\left( \left\vert \frac{\partial ^{2}f}{\partial
t\partial s}\left( x,y\right) \right\vert ^{q}+\left\vert \frac{\partial
^{2}f}{\partial t\partial s}\left( x,c\right) \right\vert ^{q}+\left\vert 
\frac{\partial ^{2}f}{\partial t\partial s}\left( a,y\right) \right\vert
^{q}+\left\vert \frac{\partial ^{2}f}{\partial t\partial s}\left( a,c\right)
\right\vert ^{q}\right) ^{\frac{1}{q}}  \notag
\end{eqnarray}

and similarly, we get%
\begin{eqnarray}
&&\left( \int\limits_{0}^{1}\int\limits_{0}^{1}\left\vert \frac{\partial
^{2}f}{\partial t\partial s}\left( tx+\left( 1-t\right) a,sy+\left(
1-s\right) d\right) \right\vert ^{q}dsdt\right) ^{\frac{1}{q}}  \label{m3} \\
&\leq &\frac{1}{2^{\frac{2}{q}}}\left( \left\vert \frac{\partial ^{2}f}{%
\partial t\partial s}\left( x,y\right) \right\vert ^{q}+\left\vert \frac{%
\partial ^{2}f}{\partial t\partial s}\left( x,d\right) \right\vert
^{q}+\left\vert \frac{\partial ^{2}f}{\partial t\partial s}\left( a,y\right)
\right\vert ^{q}+\left\vert \frac{\partial ^{2}f}{\partial t\partial s}%
\left( a,d\right) \right\vert ^{q}\right) ^{\frac{1}{q}},  \notag
\end{eqnarray}%
\begin{eqnarray}
&&\left( \int\limits_{0}^{1}\int\limits_{0}^{1}\left\vert \frac{\partial
^{2}f}{\partial t\partial s}\left( tx+\left( 1-t\right) b,sy+\left(
1-s\right) c\right) \right\vert ^{q}dsdt\right) ^{\frac{1}{q}}  \label{m4} \\
&\leq &\frac{1}{2^{\frac{2}{q}}}\left( \left\vert \frac{\partial ^{2}f}{%
\partial t\partial s}\left( x,y\right) \right\vert ^{q}+\left\vert \frac{%
\partial ^{2}f}{\partial t\partial s}\left( x,c\right) \right\vert
^{q}+\left\vert \frac{\partial ^{2}f}{\partial t\partial s}\left( b,y\right)
\right\vert ^{q}+\left\vert \frac{\partial ^{2}f}{\partial t\partial s}%
\left( b,c\right) \right\vert ^{q}\right) ^{\frac{1}{q}},  \notag
\end{eqnarray}%
\begin{eqnarray}
&&\left( \int\limits_{0}^{1}\int\limits_{0}^{1}\left\vert \frac{\partial
^{2}f}{\partial t\partial s}\left( tx+\left( 1-t\right) b,sy+\left(
1-s\right) d\right) \right\vert ^{q}dsdt.\right) ^{\frac{1}{q}}  \label{m5}
\\
&\leq &\frac{1}{2^{\frac{2}{q}}}\left( \left\vert \frac{\partial ^{2}f}{%
\partial t\partial s}\left( x,y\right) \right\vert ^{q}+\left\vert \frac{%
\partial ^{2}f}{\partial t\partial s}\left( x,d\right) \right\vert
^{q}+\left\vert \frac{\partial ^{2}f}{\partial t\partial s}\left( b,y\right)
\right\vert ^{q}+\left\vert \frac{\partial ^{2}f}{\partial t\partial s}%
\left( b,d\right) \right\vert ^{q}\right) ^{\frac{1}{q}}.  \notag
\end{eqnarray}%
Then by using the inequalities (\ref{m2})-($\ref{m5})$ in (\ref{m1}), we
obtain%
\begin{eqnarray*}
&&\left\vert A+\frac{1}{\left( b-a\right) \left( d-c\right) }%
\int\limits_{a}^{b}\int\limits_{c}^{d}f\left( u,v\right) dudv\right\vert \\
&\leq &\frac{1}{\left( p+1\right) ^{\frac{2}{p}}}\frac{1}{2^{\frac{2}{q}}}%
\times \\
&&\left\{ \frac{\left( x-a\right) ^{2}\left( y-c\right) ^{2}}{\left(
b-a\right) \left( d-c\right) }\left( \left\vert \frac{\partial ^{2}f}{%
\partial t\partial s}\left( x,y\right) \right\vert ^{q}+\left\vert \frac{%
\partial ^{2}f}{\partial t\partial s}\left( x,c\right) \right\vert
^{q}+\left\vert \frac{\partial ^{2}f}{\partial t\partial s}\left( a,y\right)
\right\vert ^{q}+\left\vert \frac{\partial ^{2}f}{\partial t\partial s}%
\left( a,c\right) \right\vert ^{q}\right) ^{\frac{1}{q}}\right. \\
&&\left. +\frac{\left( x-a\right) ^{2}\left( d-y\right) ^{2}}{\left(
b-a\right) \left( d-c\right) }\left( \left\vert \frac{\partial ^{2}f}{%
\partial t\partial s}\left( x,y\right) \right\vert ^{q}+\left\vert \frac{%
\partial ^{2}f}{\partial t\partial s}\left( x,d\right) \right\vert
^{q}+\left\vert \frac{\partial ^{2}f}{\partial t\partial s}\left( a,y\right)
\right\vert ^{q}+\left\vert \frac{\partial ^{2}f}{\partial t\partial s}%
\left( a,d\right) \right\vert ^{q}\right) ^{\frac{1}{q}}\right. \\
&&\left. +\frac{\left( b-x\right) ^{2}\left( y-c\right) ^{2}}{\left(
b-a\right) \left( d-c\right) }\left( \left\vert \frac{\partial ^{2}f}{%
\partial t\partial s}\left( x,y\right) \right\vert ^{q}+\left\vert \frac{%
\partial ^{2}f}{\partial t\partial s}\left( x,c\right) \right\vert
^{q}+\left\vert \frac{\partial ^{2}f}{\partial t\partial s}\left( b,y\right)
\right\vert ^{q}+\left\vert \frac{\partial ^{2}f}{\partial t\partial s}%
\left( b,c\right) \right\vert ^{q}\right) ^{\frac{1}{q}}\right. \\
&&\left. +\frac{\left( b-x\right) ^{2}\left( d-y\right) ^{2}}{\left(
b-a\right) \left( d-c\right) }\left( \left\vert \frac{\partial ^{2}f}{%
\partial t\partial s}\left( x,y\right) \right\vert ^{q}+\left\vert \frac{%
\partial ^{2}f}{\partial t\partial s}\left( x,d\right) \right\vert
^{q}+\left\vert \frac{\partial ^{2}f}{\partial t\partial s}\left( b,y\right)
\right\vert ^{q}+\left\vert \frac{\partial ^{2}f}{\partial t\partial s}%
\left( b,d\right) \right\vert ^{q}\right) ^{\frac{1}{q}}\right\}
\end{eqnarray*}%
which completes the proof.
\end{proof}

\begin{corollary}
\bigskip \label{c2}(1) Under the assumptions of Theorem \ref{t2}, if we
choose $x=a,$ $y=c,$ or $x=b,$ $y=d,$ we obtain the following inequality;%
\begin{eqnarray*}
&&\frac{1}{\left( b-a\right) \left( d-c\right) }\left\vert f\left(
b,d\right) -\left( b-a\right) \int\limits_{c}^{d}f\left( b,v\right)
dv-\left( d-c\right) \int_{a}^{b}f\left( u,d\right)
du+\int\limits_{a}^{b}\int\limits_{c}^{d}f\left( u,v\right) dudv\right\vert
\\
&\leq &\frac{1}{\left( b-a\right) \left( d-c\right) \left( p+1\right) ^{%
\frac{2}{p}}2^{\frac{2}{q}}}\left( \left\vert \frac{\partial ^{2}f}{\partial
t\partial s}\left( a,c\right) \right\vert ^{q}+\left\vert \frac{\partial
^{2}f}{\partial t\partial s}\left( a,d\right) \right\vert ^{q}+\left\vert 
\frac{\partial ^{2}f}{\partial t\partial s}\left( b,c\right) \right\vert
^{q}+\left\vert \frac{\partial ^{2}f}{\partial t\partial s}\left( b,d\right)
\right\vert ^{q}\right) ^{\frac{1}{q}}
\end{eqnarray*}%
(2)Under the assumptions of Theorem \ref{t2}, if we choose $x=b,$ $y=d,$ we
obtain the following inequality;%
\begin{eqnarray*}
&&\frac{1}{\left( b-a\right) \left( d-c\right) }\left\vert f\left(
a,c\right) -\left( b-a\right) \int\limits_{c}^{d}f\left( a,v\right)
dv-\left( d-c\right) \int\limits_{a}^{b}f\left( u,c\right)
du+\int\limits_{a}^{b}\int\limits_{c}^{d}f\left( u,v\right) dudv\right\vert
\\
&\leq &\frac{1}{\left( b-a\right) \left( d-c\right) \left( p+1\right) ^{%
\frac{2}{p}}2^{\frac{2}{q}}}\left( \left\vert \frac{\partial ^{2}f}{\partial
t\partial s}\left( b,d\right) \right\vert ^{q}+\left\vert \frac{\partial
^{2}f}{\partial t\partial s}\left( b,c\right) \right\vert ^{q}+\left\vert 
\frac{\partial ^{2}f}{\partial t\partial s}\left( a,d\right) \right\vert
^{q}+\left\vert \frac{\partial ^{2}f}{\partial t\partial s}\left( a,c\right)
\right\vert ^{q}\right) ^{\frac{1}{q}}
\end{eqnarray*}%
(3)Under the assumptions of Theorem \ref{t2}, if we choose $x=a,$ $y=d,$ we
obtain the following inequality;%
\begin{eqnarray*}
&&\frac{1}{\left( b-a\right) \left( d-c\right) }\left\vert f\left(
b,c\right) -\left( b-a\right) \int\limits_{c}^{d}f\left( b,v\right)
dv-\left( d-c\right) \int\limits_{a}^{b}f\left( u,c\right)
du+\int\limits_{a}^{b}\int\limits_{c}^{d}f\left( u,v\right) dudv\right\vert
\\
&\leq &\frac{1}{\left( b-a\right) \left( d-c\right) \left( p+1\right) ^{%
\frac{2}{p}}2^{\frac{2}{q}}}\left( \left\vert \frac{\partial ^{2}f}{\partial
t\partial s}\left( a,d\right) \right\vert ^{q}+\left\vert \frac{\partial
^{2}f}{\partial t\partial s}\left( a,c\right) \right\vert ^{q}+\left\vert 
\frac{\partial ^{2}f}{\partial t\partial s}\left( b,d\right) \right\vert
^{q}+\left\vert \frac{\partial ^{2}f}{\partial t\partial s}\left( b,c\right)
\right\vert ^{q}\right) ^{\frac{1}{q}}
\end{eqnarray*}%
(4)Under the assumptions of Theorem \ref{t2}, if we choose $x=b,$ $y=c,$ we
obtain the following inequality;%
\begin{eqnarray*}
&&\frac{1}{\left( b-a\right) \left( d-c\right) }\left\vert f\left(
a,d\right) -\left( b-a\right) \int\limits_{c}^{d}f\left( a,v\right)
dv-\left( d-c\right) \int\limits_{a}^{b}f\left( u,d\right)
du+\int\limits_{a}^{b}\int\limits_{c}^{d}f\left( u,v\right) dudv\right\vert
\\
&\leq &\frac{1}{\left( b-a\right) \left( d-c\right) \left( p+1\right) ^{%
\frac{2}{p}}2^{\frac{2}{q}}}\left( \left\vert \frac{\partial ^{2}f}{\partial
t\partial s}\left( b,c\right) \right\vert ^{q}+\left\vert \frac{\partial
^{2}f}{\partial t\partial s}\left( b,d\right) \right\vert ^{q}+\left\vert 
\frac{\partial ^{2}f}{\partial t\partial s}\left( a,c\right) \right\vert
^{q}+\left\vert \frac{\partial ^{2}f}{\partial t\partial s}\left( a,d\right)
\right\vert ^{q}\right) ^{\frac{1}{q}}
\end{eqnarray*}%
(5) Under the assumptions of Theorem \ref{t2}, if we choose $x=\frac{a+b}{2}%
, $ $y=\frac{c+d}{2},$ we obtain the following inequality;%
\begin{eqnarray*}
&&\left\vert \frac{f\left( a,c\right) +f\left( a,d\right) +f\left(
b,c\right) +f\left( b,d\right) }{4\left( b-a\right) \left( d-c\right) }-%
\frac{1}{2\left( d-c\right) }\int\limits_{c}^{d}f\left( a,v\right) dv-\frac{1%
}{2\left( d-c\right) }\int\limits_{c}^{d}f\left( b,v\right) dv\right. \\
&&\left. -\frac{1}{2\left( b-a\right) }\int\limits_{a}^{b}f\left( u,d\right)
du-\frac{1}{2\left( b-a\right) }\int\limits_{a}^{b}f\left( u,c\right) du+%
\frac{1}{\left( b-a\right) \left( d-c\right) }\int\limits_{a}^{b}\int%
\limits_{c}^{d}f\left( u,v\right) dudv\right\vert \\
&\leq &\frac{1}{\left( b-a\right) \left( d-c\right) 16\left( p+1\right) ^{%
\frac{2}{p}}2^{\frac{2}{q}}}\times \\
&&\left\{ \left( \left\vert \frac{\partial ^{2}f}{\partial t\partial s}%
\left( \frac{a+b}{2},\frac{c+d}{2}\right) \right\vert ^{q}+\left\vert \frac{%
\partial ^{2}f}{\partial t\partial s}\left( \frac{a+b}{2},c\right)
\right\vert ^{q}+\left\vert \frac{\partial ^{2}f}{\partial t\partial s}%
\left( a,\frac{c+d}{2}\right) \right\vert ^{q}+\left\vert \frac{\partial
^{2}f}{\partial t\partial s}\left( a,c\right) \right\vert ^{q}\right) ^{%
\frac{1}{q}}\right. \\
&&\left. +\left( \left\vert \frac{\partial ^{2}f}{\partial t\partial s}%
\left( \frac{a+b}{2},\frac{c+d}{2}\right) \right\vert ^{q}+\left\vert \frac{%
\partial ^{2}f}{\partial t\partial s}\left( \frac{a+b}{2},d\right)
\right\vert ^{q}+\left\vert \frac{\partial ^{2}f}{\partial t\partial s}%
\left( a,\frac{c+d}{2}\right) \right\vert ^{q}+\left\vert \frac{\partial
^{2}f}{\partial t\partial s}\left( a,d\right) \right\vert ^{q}\right) ^{%
\frac{1}{q}}\right. \\
&&\left. +\left( \left\vert \frac{\partial ^{2}f}{\partial t\partial s}%
\left( \frac{a+b}{2},\frac{c+d}{2}\right) \right\vert ^{q}+\left\vert \frac{%
\partial ^{2}f}{\partial t\partial s}\left( \frac{a+b}{2},c\right)
\right\vert ^{q}+\left\vert \frac{\partial ^{2}f}{\partial t\partial s}%
\left( b,\frac{c+d}{2}\right) \right\vert ^{q}+\left\vert \frac{\partial
^{2}f}{\partial t\partial s}\left( b,c\right) \right\vert ^{q}\right) ^{%
\frac{1}{q}}\right. \\
&&\left. +\left( \left\vert \frac{\partial ^{2}f}{\partial t\partial s}%
\left( \frac{a+b}{2},\frac{c+d}{2}\right) \right\vert ^{q}+\left\vert \frac{%
\partial ^{2}f}{\partial t\partial s}\left( \frac{a+b}{2},d\right)
\right\vert ^{q}+\left\vert \frac{\partial ^{2}f}{\partial t\partial s}%
\left( b,\frac{c+d}{2}\right) \right\vert ^{q}+\left\vert \frac{\partial
^{2}f}{\partial t\partial s}\left( b,d\right) \right\vert ^{q}\right) ^{%
\frac{1}{q}}\right\} .
\end{eqnarray*}
\end{corollary}

\begin{theorem}
Let $f:\Delta =\left[ a,b\right] \times \left[ c,d\right] \rightarrow 
%TCIMACRO{\U{211d} }%
%BeginExpansion
\mathbb{R}
%EndExpansion
$ be a partial differentiable mapping on $\Delta =\left[ a,b\right] \times %
\left[ c,d\right] $ and $\frac{\partial ^{2}f}{\partial t\partial s}\in
L\left( \Delta \right) $. If $\left\vert \frac{\partial ^{2}f}{\partial
t\partial s}\right\vert ^{q},q\geq 1,$ is a convex function on the
co-ordinates on $\Delta ,$ then the following inequality holds;%
\begin{eqnarray*}
&&\left\vert A+\frac{1}{\left( b-a\right) \left( d-c\right) }%
\int\limits_{a}^{b}\int\limits_{c}^{d}f\left( u,v\right) dudv\right\vert \\
\leq \left( \frac{1}{4}\right) ^{1-\frac{1}{q}} &&\left\{ K\left( \frac{1}{36%
}\left\vert \frac{\partial ^{2}f}{\partial t\partial s}\left( x,y\right)
\right\vert ^{q}+\frac{1}{18}\left\vert \frac{\partial ^{2}f}{\partial
t\partial s}\left( x,c\right) \right\vert ^{q}+\frac{1}{18}\left\vert \frac{%
\partial ^{2}f}{\partial t\partial s}\left( a,y\right) \right\vert ^{q}+%
\frac{1}{9}\left\vert \frac{\partial ^{2}f}{\partial t\partial s}\left(
a,c\right) \right\vert ^{q}\right) ^{\frac{1}{q}}\right. \\
&&+L\left( \frac{1}{36}\left\vert \frac{\partial ^{2}f}{\partial t\partial s}%
\left( x,y\right) \right\vert ^{q}+\frac{1}{18}\left\vert \frac{\partial
^{2}f}{\partial t\partial s}\left( x,d\right) \right\vert ^{q}+\frac{1}{18}%
\left\vert \frac{\partial ^{2}f}{\partial t\partial s}\left( a,y\right)
\right\vert ^{q}+\frac{1}{9}\left\vert \frac{\partial ^{2}f}{\partial
t\partial s}\left( a,d\right) \right\vert ^{q}\right) ^{\frac{1}{q}} \\
&&\left. +M\left( \frac{1}{36}\left\vert \frac{\partial ^{2}f}{\partial
t\partial s}\left( x,y\right) \right\vert ^{q}+\frac{1}{18}\left\vert \frac{%
\partial ^{2}f}{\partial t\partial s}\left( x,c\right) \right\vert ^{q}+%
\frac{1}{18}\left\vert \frac{\partial ^{2}f}{\partial t\partial s}\left(
b,y\right) \right\vert ^{q}+\frac{1}{9}\left\vert \frac{\partial ^{2}f}{%
\partial t\partial s}\left( b,c\right) \right\vert ^{q}\right) ^{\frac{1}{q}%
}\right. \\
&&\left. +N\left( \frac{1}{36}\left\vert \frac{\partial ^{2}f}{\partial
t\partial s}\left( x,y\right) \right\vert ^{q}+\frac{1}{18}\left\vert \frac{%
\partial ^{2}f}{\partial t\partial s}\left( x,d\right) \right\vert ^{q}+%
\frac{1}{18}\left\vert \frac{\partial ^{2}f}{\partial t\partial s}\left(
b,y\right) \right\vert ^{q}+\frac{1}{9}\left\vert \frac{\partial ^{2}f}{%
\partial t\partial s}\left( b,d\right) \right\vert ^{q}\right) ^{\frac{1}{q}%
}\right\}
\end{eqnarray*}%
where%
\begin{eqnarray*}
K &=&\frac{\left( x-a\right) ^{2}\left( y-c\right) ^{2}}{\left( b-a\right)
\left( d-c\right) } \\
L &=&\frac{\left( x-a\right) ^{2}\left( d-y\right) ^{2}}{\left( b-a\right)
\left( d-c\right) } \\
M &=&\frac{\left( b-x\right) ^{2}\left( y-c\right) ^{2}}{\left( b-a\right)
\left( d-c\right) } \\
N &=&\frac{\left( b-x\right) ^{2}\left( d-y\right) ^{2}}{\left( b-a\right)
\left( d-c\right) }.
\end{eqnarray*}
\end{theorem}

\begin{proof}
From Lemma \ref{l1}, we have%
\begin{eqnarray*}
&&\left\vert A+\frac{1}{\left( b-a\right) \left( d-c\right) }%
\int\limits_{a}^{b}\int\limits_{c}^{d}f\left( u,v\right) dudv\right\vert \\
&\leq &\frac{\left( x-a\right) ^{2}\left( y-c\right) ^{2}}{\left( b-a\right)
\left( d-c\right) }\int\limits_{0}^{1}\int\limits_{0}^{1}\left\vert \left(
t-1\right) \left( s-1\right) \right\vert \left\vert \frac{\partial ^{2}f}{%
\partial t\partial s}\left( tx+\left( 1-t\right) a,sy+\left( 1-s\right)
c\right) \right\vert dsdt \\
&&+\frac{\left( x-a\right) ^{2}\left( d-y\right) ^{2}}{\left( b-a\right)
\left( d-c\right) }\int\limits_{0}^{1}\int\limits_{0}^{1}\left\vert \left(
t-1\right) \left( 1-s\right) \right\vert \left\vert \frac{\partial ^{2}f}{%
\partial t\partial s}\left( tx+\left( 1-t\right) a,sy+\left( 1-s\right)
d\right) \right\vert dsdt \\
&&+\frac{\left( b-x\right) ^{2}\left( y-c\right) ^{2}}{\left( b-a\right)
\left( d-c\right) }\int\limits_{0}^{1}\int\limits_{0}^{1}\left\vert \left(
1-t\right) \left( s-1\right) \right\vert \left\vert \frac{\partial ^{2}f}{%
\partial t\partial s}\left( tx+\left( 1-t\right) b,sy+\left( 1-s\right)
c\right) \right\vert dsdt \\
&&+\frac{\left( b-x\right) ^{2}\left( d-y\right) ^{2}}{\left( b-a\right)
\left( d-c\right) }\int\limits_{0}^{1}\int\limits_{0}^{1}\left\vert \left(
1-t\right) \left( 1-s\right) \right\vert \left\vert \frac{\partial ^{2}f}{%
\partial t\partial s}\left( tx+\left( 1-t\right) b,sy+\left( 1-s\right)
d\right) \right\vert dsdt.
\end{eqnarray*}%
By using the well known Power mean inequality for double integrals, then one
has:%
\begin{eqnarray}
&&\left\vert A+\frac{1}{\left( b-a\right) \left( d-c\right) }%
\int\limits_{a}^{b}\int\limits_{c}^{d}f\left( u,v\right) dudv\right\vert
\label{5} \\
&\leq &\frac{\left( x-a\right) ^{2}\left( y-c\right) ^{2}}{\left( b-a\right)
\left( d-c\right) }\left( \int\limits_{0}^{1}\int\limits_{0}^{1}\left\vert
\left( t-1\right) \left( s-1\right) \right\vert dsdt\right) ^{1-\frac{1}{q}}
\notag \\
&&\times \left( \int\limits_{0}^{1}\int\limits_{0}^{1}\left\vert \left(
t-1\right) \left( s-1\right) \right\vert \left\vert \frac{\partial ^{2}f}{%
\partial t\partial s}\left( tx+\left( 1-t\right) a,sy+\left( 1-s\right)
c\right) \right\vert ^{q}dsdt\right) ^{\frac{1}{q}}  \notag \\
&&+\frac{\left( x-a\right) ^{2}\left( d-y\right) ^{2}}{\left( b-a\right)
\left( d-c\right) }\left( \int\limits_{0}^{1}\int\limits_{0}^{1}\left\vert
\left( t-1\right) \left( 1-s\right) \right\vert dsdt\right) ^{1-\frac{1}{q}}
\notag \\
&&\times \left( \int\limits_{0}^{1}\int\limits_{0}^{1}\left\vert \left(
t-1\right) \left( 1-s\right) \right\vert \left\vert \frac{\partial ^{2}f}{%
\partial t\partial s}\left( tx+\left( 1-t\right) a,sy+\left( 1-s\right)
d\right) \right\vert ^{q}dsdt\right) ^{\frac{1}{q}}  \notag \\
&&+\frac{\left( b-x\right) ^{2}\left( y-c\right) ^{2}}{\left( b-a\right)
\left( d-c\right) }\left( \int\limits_{0}^{1}\int\limits_{0}^{1}\left\vert
\left( 1-t\right) \left( s-1\right) \right\vert dsdt\right) ^{1-\frac{1}{q}}
\notag \\
&&\times \left( \int\limits_{0}^{1}\int\limits_{0}^{1}\left\vert \left(
1-t\right) \left( s-1\right) \right\vert \left\vert \frac{\partial ^{2}f}{%
\partial t\partial s}\left( tx+\left( 1-t\right) b,sy+\left( 1-s\right)
c\right) \right\vert ^{q}dsdt\right) ^{\frac{1}{q}}  \notag \\
&&+\frac{\left( b-x\right) ^{2}\left( d-y\right) ^{2}}{\left( b-a\right)
\left( d-c\right) }\left( \int\limits_{0}^{1}\int\limits_{0}^{1}\left\vert
\left( 1-t\right) \left( 1-s\right) \right\vert dsdt\right) ^{1-\frac{1}{q}}
\notag \\
&&\times \left( \int\limits_{0}^{1}\int\limits_{0}^{1}\left\vert \left(
1-t\right) \left( 1-s\right) \right\vert \left\vert \frac{\partial ^{2}f}{%
\partial t\partial s}\left( tx+\left( 1-t\right) b,sy+\left( 1-s\right)
d\right) \right\vert ^{q}dsdt.\right) ^{\frac{1}{q}}  \notag
\end{eqnarray}

Since $\left\vert \frac{\partial ^{2}f}{\partial t\partial s}\right\vert
^{q} $ is convex function on the co-ordinates on $\Delta $, we know that for 
$t,s\in \left[ 0,1\right] $%
\begin{eqnarray*}
&&\left\vert \frac{\partial ^{2}f}{\partial t\partial s}\left( tx+\left(
1-t\right) a,sy+\left( 1-s\right) c\right) \right\vert ^{q} \\
&\leq &t\left\vert \frac{\partial ^{2}f}{\partial t\partial s}\left(
x,sy+\left( 1-s\right) c\right) \right\vert ^{q}+\left( 1-t\right)
\left\vert \frac{\partial ^{2}f}{\partial t\partial s}\left( a,sy+\left(
1-s\right) c\right) \right\vert ^{q} \\
&\leq &t\left( s\left\vert \frac{\partial ^{2}f}{\partial t\partial s}\left(
x,y\right) \right\vert ^{q}+\left( 1-s\right) \left\vert \frac{\partial ^{2}f%
}{\partial t\partial s}\left( x,c\right) \right\vert ^{q}\right) +\left(
1-t\right) \left( s\left\vert \frac{\partial ^{2}f}{\partial t\partial s}%
\left( a,y\right) \right\vert ^{q}+\left( 1-s\right) \left\vert \frac{%
\partial ^{2}f}{\partial t\partial s}\left( a,c\right) \right\vert
^{q}\right)
\end{eqnarray*}%
and by using the fact that%
\begin{equation*}
\left( \int\limits_{0}^{1}\int\limits_{0}^{1}\left\vert \left( t-1\right)
\left( s-1\right) \right\vert dsdt\right) ^{1-\frac{1}{q}}=\left( \frac{1}{4}%
\right) ^{1-\frac{1}{q}}
\end{equation*}%
we get%
\begin{eqnarray}
&&\left( \int\limits_{0}^{1}\int\limits_{0}^{1}\left\vert \left( t-1\right)
\left( s-1\right) \right\vert \left\vert \frac{\partial ^{2}f}{\partial
t\partial s}\left( tx+\left( 1-t\right) a,sy+\left( 1-s\right) c\right)
\right\vert ^{q}dsdt\right) ^{\frac{1}{q}}  \label{1} \\
&\leq &\left( \int\limits_{0}^{1}\int\limits_{0}^{1}\left\vert \left(
t-1\right) \left( s-1\right) \right\vert \left\{ ts\left\vert \frac{\partial
^{2}f}{\partial t\partial s}\left( x,y\right) \right\vert ^{q}+t\left(
1-s\right) \left\vert \frac{\partial ^{2}f}{\partial t\partial s}\left(
x,c\right) \right\vert ^{q}\right. \right.  \notag \\
&&\left. \left. +\left( 1-t\right) s\left\vert \frac{\partial ^{2}f}{%
\partial t\partial s}\left( a,y\right) \right\vert ^{q}+\left( 1-t\right)
\left( 1-s\right) \left\vert \frac{\partial ^{2}f}{\partial t\partial s}%
\left( a,c\right) \right\vert ^{q}\right\} dtds\right) ^{\frac{1}{q}}  \notag
\\
&=&\left( \frac{1}{36}\left\vert \frac{\partial ^{2}f}{\partial t\partial s}%
\left( x,y\right) \right\vert ^{q}+\frac{1}{18}\left\vert \frac{\partial
^{2}f}{\partial t\partial s}\left( x,c\right) \right\vert ^{q}+\frac{1}{18}%
\left\vert \frac{\partial ^{2}f}{\partial t\partial s}\left( a,y\right)
\right\vert ^{q}+\frac{1}{9}\left\vert \frac{\partial ^{2}f}{\partial
t\partial s}\left( a,c\right) \right\vert ^{q}\right) ^{\frac{1}{q}}  \notag
\end{eqnarray}

and similarly, we get%
\begin{eqnarray}
&&\left( \int\limits_{0}^{1}\int\limits_{0}^{1}\left\vert \frac{\partial
^{2}f}{\partial t\partial s}\left( tx+\left( 1-t\right) a,sy+\left(
1-s\right) d\right) \right\vert ^{q}dsdt\right) ^{\frac{1}{q}}  \label{2} \\
&\leq &\left( \frac{1}{36}\left\vert \frac{\partial ^{2}f}{\partial
t\partial s}\left( x,y\right) \right\vert ^{q}+\frac{1}{18}\left\vert \frac{%
\partial ^{2}f}{\partial t\partial s}\left( x,d\right) \right\vert ^{q}+%
\frac{1}{18}\left\vert \frac{\partial ^{2}f}{\partial t\partial s}\left(
a,y\right) \right\vert ^{q}+\frac{1}{9}\left\vert \frac{\partial ^{2}f}{%
\partial t\partial s}\left( a,d\right) \right\vert ^{q}\right) ^{\frac{1}{q}%
},  \notag
\end{eqnarray}%
\begin{eqnarray}
&&\left( \int\limits_{0}^{1}\int\limits_{0}^{1}\left\vert \frac{\partial
^{2}f}{\partial t\partial s}\left( tx+\left( 1-t\right) b,sy+\left(
1-s\right) c\right) \right\vert ^{q}dsdt\right) ^{\frac{1}{q}}  \label{3} \\
&\leq &\left( \frac{1}{36}\left\vert \frac{\partial ^{2}f}{\partial
t\partial s}\left( x,y\right) \right\vert ^{q}+\frac{1}{18}\left\vert \frac{%
\partial ^{2}f}{\partial t\partial s}\left( x,c\right) \right\vert ^{q}+%
\frac{1}{18}\left\vert \frac{\partial ^{2}f}{\partial t\partial s}\left(
b,y\right) \right\vert ^{q}+\frac{1}{9}\left\vert \frac{\partial ^{2}f}{%
\partial t\partial s}\left( b,c\right) \right\vert ^{q}\right) ^{\frac{1}{q}%
},  \notag
\end{eqnarray}%
\begin{eqnarray}
&&\left( \int\limits_{0}^{1}\int\limits_{0}^{1}\left\vert \frac{\partial
^{2}f}{\partial t\partial s}\left( tx+\left( 1-t\right) b,sy+\left(
1-s\right) d\right) \right\vert ^{q}dsdt.\right) ^{\frac{1}{q}}  \label{4} \\
&\leq &\left( \frac{1}{36}\left\vert \frac{\partial ^{2}f}{\partial
t\partial s}\left( x,y\right) \right\vert ^{q}+\frac{1}{18}\left\vert \frac{%
\partial ^{2}f}{\partial t\partial s}\left( x,d\right) \right\vert ^{q}+%
\frac{1}{18}\left\vert \frac{\partial ^{2}f}{\partial t\partial s}\left(
b,y\right) \right\vert ^{q}+\frac{1}{9}\left\vert \frac{\partial ^{2}f}{%
\partial t\partial s}\left( b,d\right) \right\vert ^{q}\right) ^{\frac{1}{q}%
}.  \notag
\end{eqnarray}%
Then by using the inequalities (\ref{1})-($\ref{4})$ in (\ref{5}), we obtain%
\begin{eqnarray*}
&&\left\vert A+\frac{1}{\left( b-a\right) \left( d-c\right) }%
\int\limits_{a}^{b}\int\limits_{c}^{d}f\left( u,v\right) dudv\right\vert \\
\leq \left( \frac{1}{4}\right) ^{1-\frac{1}{q}} &&\left\{ K\left( \frac{1}{36%
}\left\vert \frac{\partial ^{2}f}{\partial t\partial s}\left( x,y\right)
\right\vert ^{q}+\frac{1}{18}\left\vert \frac{\partial ^{2}f}{\partial
t\partial s}\left( x,c\right) \right\vert ^{q}+\frac{1}{18}\left\vert \frac{%
\partial ^{2}f}{\partial t\partial s}\left( a,y\right) \right\vert ^{q}+%
\frac{1}{9}\left\vert \frac{\partial ^{2}f}{\partial t\partial s}\left(
a,c\right) \right\vert ^{q}\right) ^{\frac{1}{q}}\right. \\
&&+L\left( \frac{1}{36}\left\vert \frac{\partial ^{2}f}{\partial t\partial s}%
\left( x,y\right) \right\vert ^{q}+\frac{1}{18}\left\vert \frac{\partial
^{2}f}{\partial t\partial s}\left( x,d\right) \right\vert ^{q}+\frac{1}{18}%
\left\vert \frac{\partial ^{2}f}{\partial t\partial s}\left( a,y\right)
\right\vert ^{q}+\frac{1}{9}\left\vert \frac{\partial ^{2}f}{\partial
t\partial s}\left( a,d\right) \right\vert ^{q}\right) ^{\frac{1}{q}} \\
&&\left. +M\left( \frac{1}{36}\left\vert \frac{\partial ^{2}f}{\partial
t\partial s}\left( x,y\right) \right\vert ^{q}+\frac{1}{18}\left\vert \frac{%
\partial ^{2}f}{\partial t\partial s}\left( x,c\right) \right\vert ^{q}+%
\frac{1}{18}\left\vert \frac{\partial ^{2}f}{\partial t\partial s}\left(
b,y\right) \right\vert ^{q}+\frac{1}{9}\left\vert \frac{\partial ^{2}f}{%
\partial t\partial s}\left( b,c\right) \right\vert ^{q}\right) ^{\frac{1}{q}%
}\right. \\
&&\left. +N\left( \frac{1}{36}\left\vert \frac{\partial ^{2}f}{\partial
t\partial s}\left( x,y\right) \right\vert ^{q}+\frac{1}{18}\left\vert \frac{%
\partial ^{2}f}{\partial t\partial s}\left( x,d\right) \right\vert ^{q}+%
\frac{1}{18}\left\vert \frac{\partial ^{2}f}{\partial t\partial s}\left(
b,y\right) \right\vert ^{q}+\frac{1}{9}\left\vert \frac{\partial ^{2}f}{%
\partial t\partial s}\left( b,d\right) \right\vert ^{q}\right) ^{\frac{1}{q}%
}\right\}
\end{eqnarray*}%
which completes the proof.
\end{proof}


\begin{thebibliography}{9}
\bibitem{SS} S.S. Dragomir, On Hadamard's inequality for convex functions on
the co-ordinates in a rectangle from the plane, Taiwanese Journal of Math.,
5, 2001, 775-788.

\bibitem{ALO} M. Alomari and M. Darus, On the Hadamard's inequality for $%
\log -$convex functions on the co-ordinates, Journal of Inequalities and
Appl., 2009, article ID 283147.

\bibitem{BAK} M.K. Bakula and J. Pe\v{c}ari\'{c}, On the Jensen's inequality
for convex functions on the co-ordinates in a rectangle from the plane,
Taiwanese Journal of Math., 5, 2006, 1271-1292.

\bibitem{OZ} M.E. \"{O}zdemir, E. Set, M.Z. Sar\i kaya, Some new Hadamard's
type inequalities for co-ordinated $m-$convex and ($\alpha ,m)-$convex
functions, Hacettepe J. of. Math. and St., 40, 219-229, (2011).

\bibitem{OZ2} M.Z. Sar\i kaya, E. Set, M. Emin \"{O}zdemir and S.S.
Dragomir, New some Hadamard's type inequalities for co-ordinated convex
functions, Accepted.

\bibitem{A} M. Emin \"{O}zdemir, Havva Kavurmac\i , Ahmet Ocak Akdemir and
Merve Avc\i , Inequalities for convex and $s-$convex functions on $\Delta
=[a,b]x[c,d]$, Journal of Inequalities and Applications 2012:20,
doi:10.1186/1029-242X-2012-20.

\bibitem{B} M. Emin \"{O}zdemir, M. Amer Latif and Ahmet Ocak Akdemir, On
some Hadamard-type inequalities for product of two $s-$convex functions on
the co-ordinates, Journal of Inequalities and Applications, 2012:21,
doi:10.1186/1029-242X-2012-21.
\end{thebibliography}
\end{document}